\documentclass{amsart}
\usepackage{amsmath,amsthm}
\usepackage{amsfonts,amssymb}

\usepackage{enumerate}

\newcommand{\newchi}{\ensuremath \raisebox{2pt}{$\chi$}}

\def\XXint#1#2#3{{\setbox0=\hbox{$#1{#2#3}{\int}$ }
\vcenter{\hbox{$#2#3$ }}\kern-.6\wd0}}

\def\Ld{\Lambda}

\newtheorem{thm}{Theorem}[section]
\newtheorem{cor}[thm]{Corollary}
\newtheorem{lem}[thm]{Lemma}
\newtheorem{prop}[thm]{Proposition}
\newtheorem{exam}[thm]{Example}

\newtheorem{defn}[thm]{Definition}

\theoremstyle{remark}
\newtheorem{rem}{Remark}[section]

\numberwithin{equation}{section}

 \def\tr{{\triangle}}

\def\f{\frac}
\def\va{\varepsilon}
\def\Bl{\Bigl}
\def\Br{\Bigr}
 \def\a{{\alpha}}
 \def\b{{\beta}}
 \def\g{{\gamma}}
 
 \def\t{{\theta}}

 \def\s{{\sigma}}
 \def\la{{\langle}}
 \def\ra{{\rangle}}

 \def\CC{{\mathbb C}}
 
 \def\NN{{\mathbb N}}

 \def\RR{{\mathbb R}}
  
 \def\ZZ{{\mathbb Z}}

        \def\vi{\varphi}
        \def\p{\partial}

\newcommand{\wt}{\widetilde}
\newcommand{\wh}{\widehat}

\newcommand{\al}{\alpha}
\def\be{\beta}
\def\bl{\bigl}
\def\br{\bigr}

\def\Ld{\Lambda}

\def\Bl{\Bigl}
\def\Br{\Bigr}
\def\f{\frac}
\def\({\Bigl(}
\def \){ \Bigr)}

\def\ga{\gamma}

\def\hb{\hfill$\Box$}

\def\sa{\sigma}

\def\DD{\mathcal{D}}

\def\VV{\mathcal{V}}

\begin{document}

\title[]{Compression using quasi-interpolation} 
   \author{ Martin D. Buhmann}

   \address{Mathematics Institute,
   Justus-Liebig University\\
    Arndtstrasse 2\\
     D-35392 Giessen, Germany}
\email{buhmann@math.uni-giessen.de}

\author{Feng Dai}
\address{Department of Mathematical and Statistical Sciences\\
University of Alberta\\ Edmonton, Alberta T6G 2G1, Canada.}
\email{fdai@ualberta.ca}
\thanks{The work of the second  author was  supported  in part by NSERC  Canada under
grant RGPIN 311678-2010 and by the Alexander-von-Humboldt Foundation.}

\begin{abstract}
We consider quasi-interpolation with a main application
in radial basis function approximations and compression in this article.
Constructing and using these quasi-interpolants, we  consider wavelet and compression-type
  approximations from their linear spaces and provide convergence
  estimates. The results include an error estimate for nonlinear
  approximation by 
  quasi-interpolation,
  results about compression in the space of
  continuous functions and a pointwise convergence estimate for
  approximands of low smoothness. \end{abstract}
  \maketitle

  \section{Introduction}

We wish to study multivariate approximation schemes in this artice
which are useful approximation methods called quasi-interpolation methods. A principal
useful context -- but not the only one --  we have in mind is the related to
the so-called
radial basis functions and the linear spaces spanned by their
translates because they in particular provide excellent means for multivariate approximations in
$d$-dimensional spaces. They are known in a variety of types for at
least thirty years \cite{Franke}. Among their attractive,
interesting properties are {\it (i)\/} their approximation orders \cite{Buhmann, We}
and {\it (ii)\/} their ability to adapt to many different applications
in several dimensions. Thus
they are similar to other successful approximation methods
such as multivariate splines or finite elements (both as the
aforementoned piecewise polynomials) in low dimensional ambient spaces
\cite{BS},\cite{deBHRS}. The most often used forms of
approximation are either interpolation at the same points by which
they are translated \cite{Buhmann,BDR}, or quasi-interpolation (as detailed below), or
wavelets \cite{Daubechies}, \cite{BuhmMC}, \cite{BuhmMic}.

All the different approaches to multivariate approximation
have very positive aspects; in this article we
choose among them the mentioned quasi-interpolation \cite{BuhmJAT}. To
this end, let $A$ be a set of quasi-uniformly distributed points in
$\RR^d$, let a function $f:\RR^d\to\RR$ be a at a minimum continuous
given and to be approximated (the ``approximand'') based on knowing
its values on $A$, and let $\psi_\alpha$, $\alpha\in A$, be decaying
continuous functions. The we define the quasi-interpolant in the
simplest case as
$$ Qf(x)=\sum_{\alpha\in A}f(\alpha)\psi_\alpha(x),\qquad x\in\RR^d,$$
requiring that the $\psi_\alpha$ decay sufficiently fast for large
argument,  so that the
above sum is well-defined, that is
$$ |\psi_\alpha(x)|=O((1+\|x-\alpha|)^{-d-m_2})$$
for a positive $m_2$ and agrees with $f$ in case $f$ is a
polynomial of some maximum degree $\ell$ to be specified below. 

We wish to
find some pointwise convergence
results of such approximations, and we will study compression using them too.
Our theorems include an error estimate for nonlinear
  approximation by quasi-interpolation using radial basis functions and otherwise,
  positive and negative results about compression in the space of
  continuous functions using such approximations, and a pointwise
  convergence estimate for approximands of low smoothness and quasi-interpolation.

First, still keeping in mind the key context in which we wish to
create and use quasi-interpolation, we mention a few advantages of the radial basis function
approach: first they are available in any dimension of
the ambient space, so that it is now possible -- in contrast to many of the polynomial
spline approaches -- to approximate target functions in arbitrary
dimensions. This is needed often in practical scientific applications.
From time to time
they are actually the only feasible approach to function
approximation. 

In this paper, we have some pointwise convergence results for very general sets of
approximands and compression methods. After repeating some preliminary
results partly from \cite{BD1}, we analyse approximants with quasi-interpolation and find
that both for compression and for deriving the convergence properties of the approximants and the
spaces where they are from, quasi-interpolation is a suitable approach. Let us initially describe the form of
the approximants more precisely, including some examples for radial basis functions.

Since they will be our main application, let us briefly summarise that
radial basis function approximants typically have the form of sums
multiples of translates  
$$\vi(\|\cdot-\alpha\|),$$ where the norm is Euclidean, $\alpha\in\RR^d$ are
from a fixed, discrete set $A$ of points (finite in applications, often
infinite in the theoretical analysis, sometimes equally spaced),
and $\vi$ is a continuous, univariate function that is called the
radial basis function. Typical radial functions are multiquadrics
$\vi(r)=\sqrt{r^2+c^2}$ for a positive, fixed parameter $c$ (which
may incidentally be zero, in which case we speak about the linear radial
basis function $\vi(r)=r$), or $\vi(r)=r^2\log r$, the so-called
thin-plate spline \cite{Duchon}, the Dagum-Class \cite{BPDB}
$\phi(r)=1-(r^\beta/(1+r^\beta))^\gamma$ with a
variety of choices for the parameters $\beta$ and $\gamma$, or an exponential function
$\vi(r)=\exp(-c^2r^2)$ or $\vi(r)=\exp(-cr)$ with the same positive
parameter $c$ as above (but here it is not allowed to vanish).



There are many ways to form decaying functions for quasi-interpolation
as (finite) linear
combinations of (increasing) radial basis functions, see for instance
\cite{BDR}; all of them can be viewed as a type of preconditioning of
the approximation methods. For this article we adopt the
conditions with respect to general quasi-interpolation 
from \cite{BDL} as they are general and admit even
non-equally spaced data $\alpha$. The resulting localised functions are
central to pointwise-error estimates which are one of the central features of \cite{BD1} and to some extent in
this work. Estimates in $L^p$-norms such as presented in the third
section are also important, for instance, in the paper \cite{JM}
about errors of interpolation by
surface or thin-plate splines (but not about quasi-interpolation as we
do here), and in the papers
 \cite{BDR, JL, LJC} for $L^p$-approximation
  from shift-invariant spaces (where the
   quasi-interpolation used  is different from that in this paper).  Using Gau\ss-quadrature formulae and similar techniques, the
paper \cite{WX} does indeed consider quasi-interpolation and
$L^p$-error estimates too, but neither methods nor results are going in the
same direction as here.


We  use a wavelet approach
as in \cite{Daubechies}, \cite{BuhmMic}, \cite{BuhmMC}, and compression with general quasi-interpolation. 

We begin our article by summarising suitable conditions on the aforementioned
generation of localised quasi-interpolating basis functions. We will then give some useful pointwise
convergence estimates (as to be contrasted to the usual ones in norms) for approximands which may be
outside the usual smoothness requirements (of lower smoothness; or ``rougher'' in other words) of differentiability as in \cite{BL}.

The idea of compression, i.e.\ approximations employed where the sizes
of the coefficients in use are essential gives rise to the results, whereas approximations of the wavelet-type
are featured at the end of the paper. We include, among other things, a
result about the $N$-width type approximations and a negative result
about compression within the space of continuous functions.

Throughout the paper, every  function on $\RR^d$  and
every subset of $\RR^d$  are  assumed to be  Lebesgue measurable.
 The letter $C$ will  denote a general positive
constant depending only on the parameters indicated
as subscripts,
and the notation  $C'\sim C''$  means that there exist
inessential positive  constants $c_1$, $C_1$ such that $c_1
 C'\leq
C ''\leq C_1 C'$.

\section{Quasi-interpolation: General assumptions and results}

In this section, we summarise  results from \cite{BDL}, \cite{BD1} and  \cite{BuhmJAT}
on quasi-interpolation and using the central example of  radial basis functions which will play  crucial  roles   in   our discussions in the later sections of this article.
We will also at this stage set the general definitions and assumptions on the radial
basis functions $\vi$ in use which enable us to go from plain shifts of
radial basis functions to quasi-interpolants. Indeed, we no longer
require the functions to be radially symmetric as in the introduction
but deal with general, $d$-variate, real-valued continuous functions $\vi$.

We start with a brief description of some necessary notation.  The
Fourier transform 
is specified by
$$\wh{f}(\xi):=\int_{\RR^d} f(y) e^{-i\xi \cdot y}\, dy,\   \ \  \xi\in \RR^d,\   \  f\in L^1(\RR^d),$$
thus the inverse Fourier transform will be scaled in this way:
$$ {1\over(2\pi)^d}\int_{\RR^d} \wh{f}(\xi) e^{i\xi \cdot y}\,
d\xi,\   \ \  y\in \RR^d,\   \  \wh{f}\in L^1(\RR^d).$$
We provide these definitions also because we need to state which
normalisation we take for the transforms.
The  convolution of two functions on $\RR^d$ is defined by
$$f\ast g(x):=\int_{\RR^d} f(x-y) g(y)\, dy,\     \
  x\in\RR^d,$$
  where $ f\in \bigcup_{1\leq p\leq \infty}
  L^p(\RR^d)$   and
  $ g\in L^1$, or vice-versa, or  $f$ and $g$ are both nonnegative.
We denote by  $B_r(x):=\{y\in\RR^d: \|x-y\|\leq r\}$ the ball centred at $x\in\RR^d$ with radius $r$.
  For  $f\in C^k(\RR^d)$ and $x\in \RR^d$, we use the notation $T_x^k
  f$ to denote the Taylor polynomial of $f$ of  degree $k$  at the
  point $x$; that is,
$$ T_x^k f(y)=\sum_{|\ga|\leq k} \f{D^\ga f(x)}{\ga!} (y-x)^\ga,\   \ y\in
\RR^d.$$
Here, with respect to $\ga=(\ga_1, \ldots, \ga_d)\in\ZZ_+^d$ we use the standard multi-index
notation; in particular, $|\ga|=\ga_1+\cdots+\ga_d$ and $D^\ga=\Bl(\f{\p}{\p x_1}\Br)^{\ga_1}\cdots \Bl(\f {\p}{\p x_d}\Br)^{\ga_d}$.
Furthermore, we denote by $\Pi_n^d$  the set  of all algebraic polynomials  of degree at most $n$ on $\RR^d$,
$\mathcal{P}_j^d$ being the homogeneous ones of degree $j$.

The set of centres by which the quasi-interpolating basis functions are translated
is of course relevant in any approach using quasi-interpolation. In our work here, we
allow infinite sets of such centres which need not, however, be placed
on a lattice. 
 A countable subset $A$ of $\RR^d$ is called {\it quasi-uniformly distributed\/} in $\RR^d$ if there exist positive constants
$c_0$ and $C_1$ such that $\RR^d=\cup_{\alpha\in A}  B_{c_0} (\alpha)$, and $$\sum_{\alpha\in A}  \newchi_{B_{c_0}(\alpha)}\leq C_1.$$
 We denote by $\mathcal{A}$ the collection of  all quasi-uniformly distributed subsets of $\RR^d$ with the same constants $c_0$ and $C_1$.
 For convenience, we call the constants $c_0$ and $C_1$ here  the geometric constants of the set $A$.
 Throughout this paper,  the letter $A$  will always  denote a
 quasi-uniformly distributed subset of $\RR^d$ with the same and fixed
 geometric constants  $c_0$ and $C_1$.
 All  the general constants  $C$  in this paper may depend on the
 geometric constants  $c_0$ and $C_1$  but will be independent of  the
 set $A$ itself.

 The following lemma,  proved in \cite[Lemma 3]{BDL}, and used also in
 \cite{BD1}, plays an important general  role in the definition of
 quasi-interpolation:
\begin{lem}\label{lem-2-1}
Given   a quasi-uniformly distributed
subset $A$ of $\RR^d$  and a positive integer $k$,       there exists a sequence of compactly supported  functions   $\{ N_\a\}_{\a\in A}\subset C_c^{k}(\RR^d) $  with the  properties that each $N_\a$ is supported in the ball $B_{M/2}(\a)$ for  a positive constant $M$;     $\sup_{\a\in A} \|D^\g N_{\a}\|_\infty \leq C$  for  all $\ga\in\ZZ_+^d$ with  $|\ga|\leq k$;  and
$$q=\sum_{\a\in A} q(\a) N_{\a},\    \    \   \forall q\in\Pi_k^d.$$
Here the constants $C$ and $M$ depend only on $k$, $d$.
 \end{lem}
 According to the proof in \cite[Lemma 3]{BDL},
 each  function $N_\a$  in the above lemma may be a compactly supported  multivariate spline
 depending  only on $d$, $k$ and the points in the set $A\cap B_{M/2} (\a)$.   In particular, if $A\cap B_{M/2}(\a)=\ZZ^d \cap B_{M/2}(\a)$ for some $\a\in A$, then
we write $N_\a(x)=N(x-\a)$ for a compactly supported multivariate spline $N\in C_c^k (\RR^d)$ independent of the set $A$.

  Let now $\vi: \RR^d\to \RR$ be a continuous function. We require this $\vi$ to be of at most polynomial growth so that it can be viewed as a tempered
distribution or so-called generalised function. Therefore it has a  distributional Fourier
transform $\wh{\vi}$ which we specifically require to be of the form
$\wh{\vi}(\xi)=(F(\xi)+bG_0(\xi)\log\|\xi\|) /G(\xi)$
 for $\xi$ near the origin, say for  $\xi\in B_1(0)$.
Here $b$ is a constant, $G_0$ is a homogeneous polynomial of degree $n_0\geq0$. Furthermore
$F$ and $G$ are real-valued  functions on the ball $B_1(0)$
satisfying the following conditions:
\begin{enumerate}
\item[(2A)] $G$ is a homogeneous polynomial of positive even degree $2m_1$ and $G(\xi)\neq 0$ for all $\xi\in\RR^d\setminus\{0\}$.
    \item[(2B)] $F\in C^{m_0}(B_1(0))\cap C^{m_0+d+1}(B_1(0)\setminus\{0\})$ for some nonnegative integer $m_0$,  $F(0)\neq 0$,   and  there exists  $\t\in (0,1]$  such that for   all  $\ga\in\ZZ_+^d$ with $|\g|\leq m_0+d+1$,
        \begin{equation}\label{2-1-3}\Bigl|D^\ga \bigl[F(\xi)-\tau_{m_0}(\xi)\bigr]\Bigr|
\leq C \|\xi\|^{m_0+\t -|\ga|},\ \   \text{as  $\|\xi\|\to 0_+$},\end{equation}
        where $\tau_{m_0}$ denotes the Taylor polynomial of $F$ of degree $m_0$ about $\xi=0$.
      \end{enumerate}
 In addition to the  conditions (2A) and (2B), we also assume  the following  condition  on   derivatives of  $\wh{\vi}(\xi)$ for large $\|\xi\|$: \begin{enumerate}
    \item[(2C)] The Fourier transform $\wh{\vi}$ of $\vi$ is $(m_0+d+1)$-times continuously differentiable  on the domain $\{\xi\in\RR^d:\  \|\xi\|>\f14\}$,  and
             \begin{equation}\label{1-8-3}\max_{|\g|\leq m_0+d+1}\int_{\|\xi\|\ge \f14} |D^\g \wh{\vi}(\xi)|\, d\xi\leq C<\infty.
            \end{equation} \end{enumerate}

As in   \cite[p. 245]{BDL}, we introduce the following definition:
\begin{defn}  Let $T_{m_0}$ denote the Taylor polynomial of the function $1/F$ of degree $m_0$ at the origin, and let $P(x)=T_{m_0}(x) G(x)$.
For each $\a, \b\in A$, define
\begin{equation}\label {2-3-0-0}\mu_{\a\b} := \int_{\RR^d} N_\a(x) P(iD) N_\b (x)\, dx,\end{equation}
where   $\{N_\a\}_{\a\in A}$  is the sequence of functions  given in Lemma \ref{lem-2-1} with $k=2m_1+m_0+1$.
\end{defn}

From the above  definition and Lemma \ref{lem-2-1},   it is clear that
 $\mu_{\a\b}=0$ if $\|\a-\b\|\ge M$, $\mu_{\a \b} =\overline{\mu_{\b\a}}$ for all $\a,\b\in A$, where $M$ is  the positive constant given in Lemma \ref{lem-2-1} with $k=2m_1+m_0+1$.
More importantly,  the sequence $\{\mu_{\a\b}\}$  has the following important property, as was shown in \cite[Theorem 7]{BDL}:
\begin{equation}\label{3-16}\sum_{\b\in A} \mu_{\a\b} q(\b)=0, \    \  \forall \a\in A,\   \
  \forall q\in \text{ker}\, G(D) \cap \Pi_{2m_1+m_0+1}^d.\end{equation}
  Furthermore, in the case when
  $A=\ZZ^d$,  \begin{equation}\label{2-5-26}\mu_{\a\b}=\mu_{\b-\a}=\int_{\RR^d}
    N(y-\b+\a) P(iD) N(y)\,
    dy,\    \    \   \a,\b\in\ZZ^d,\end{equation}  and 
  the sequence satisfies (\cite{BuhmJAT, BDL})
  \begin{equation}\label{2-5-1-1}
  \sum_{j\in\ZZ^d} \mu_j q(j) =P(iD) q(0),\    \     \    \forall q \in \Pi_{2m_1+m_0+1}^d.
  \end{equation}
  The expression on the left-hand side of \eqref{3-16} or \eqref{2-5-1-1} is in fact  a  local   difference of the polynomial $q$
at  the point $\a\in A$ which is closely related to the derivative $P(iD) q(\a)$.

A remarkable fact is that the above defined  sequence
 admits  the mentioned quasi-interpolation, central and useful for our application.
 More precisely,  under the above  assumptions on the Fourier transform  $\wh{\vi}$,  the following result \cite{BD1} holds:\\

{\bf Theorem A.}  {\it  Under the above  assumptions,  the functions $\psi_\a$, $\a\in A$, defined by
\begin{equation}\label{2-4-rev}\psi_\a(x):= \sum_{\b\in A} \mu_{\a\b} \vi(x-\b),\   \ x\in\RR^d,\end{equation}
 have the decaying property
    \begin{equation}\label{1-1}
    |\psi_\a (x) |\leq C ( 1+\|x-\a\|)^{-d-m_2},\   \  x\in\RR^d,\    \    \a\in A,
    \end{equation}
    where
    \begin{equation}\label{2-4-3}m_2:=\begin{cases}m_0+\t,&\  \  \text{if $b=0$};\\
    \min\{m_0+\t, n_0\},&\   \  \text{if $b\neq 0$};
    \end{cases}
    \end{equation}
     and furthermore,
 the  operator defined by (``quasi-interpolation'')
\begin{equation}\label{1-2}Qf:=\sum_{\a\in A} f(\a) \psi_\a
\end{equation} reproduces all polynomials of total degree
        at most $\ell_0:=\min\{m_0, 2m_1-1\}$; that is, $Q p=p$ for all $p\in \Pi_{\ell_0}^d$.} \\

Here and elsewhere in the  paper,  the letter $C$ denotes a constant that  may  depend  on $m_1, m_0$, $n_0$, and $d$. 

\begin{rem}\begin{enumerate}[\rm (i)]
\item
We would like to remark that, in fact,
in the case of $b=0$, \mbox{Theorem A }   with slightly weaker  decaying estimates of the functions  $\psi_\a$   was proved in \cite{BDL},   where  \eqref{1-1} with $m_0$ in place of $m_2$  was  shown.
 Infinite differentiability on
$\RR^d\setminus\{0\}$   and  slightly stronger pointwise estimates of the derivatives  $D^\g\wh{\vi}(\xi)$ for large $\xi$ were required in \cite{BDL}, but  by looking carefully at the proofs there, the
assumptions can
be slightly reduced as we did above.
Concretely, the aforementioned differentiability assumptions are used
in that paper \cite{BDL} to obtain the correct asymptotic behaviour of
the inverse Fourier transform of the remainder term $\hat
r_{\ell+1+\theta} (\omega)$. 
For this, no infinite differentiability is
  required as indeed the said inverse Fourier transform need not decay
  of arbitrary high negative order but only at the same rate as the
  rest of the expression.

\item Here an   additional term $b\log \|\xi\|$ has been added  in the Fourier transform of $\vi$  in order to  include the ``shifted'' versions   of some  important radial basis functions,  where   a direct  application of the general results of \cite[Theorem 8]{BDL} normally
would  not yield  the optimal decaying estimate \eqref{1-1} .
In fact,  Theorem A of \cite{BD1} for  $b\in\CC$,  $A=\ZZ^d$ and $G(x)=\|x\|^{2m}$
was   proved   in \cite{BuhmJAT}.
\end{enumerate}
\end{rem}

The combination of localisation of the operator (by the decaying
properties of $\psi_\alpha$) and the polynomial reproduction admits
the derivation of approximation orders to sufficiently smooth
approximands. This requires, of course, a scaling of the given centres.

Indeed, given $h>0$, the following quasi-interpolation operator with scaling was introduced and studied in \cite{BDL}:
\begin{equation}\label{1-10-b}
Q_h f(x):=\sum_{\a\in A} f(h\a) \psi_\a ( h^{-1} x),\
 \ x\in \RR^d.\end{equation}

%

For the rest of the paper, we will set $\ell:=\min\{ m_0-1, 2m_1-1\}$. Observe that condition \eqref{1-2} implies that we have the
    desirable polynomial reproduction property
          $$ Q_h q=q,\     \  \forall q\in \Pi_{\ell}^d.$$
      The following estimates were proved  in \cite{BDL} for the case of  $b=0$: If $f\in C^{\ell+1}(\RR^d)$ has bounded partial derivatives of total orders $\ell$ and $\ell+1$, then for any $h\in (0,1)$,
    \begin{equation}\label{0-3}
     \|f-Q_h f \|_\infty \leq  C_f h^{\ell+1}|\log h|,\end{equation}
     where $$C_f:=\|\nabla^{\ell} f\|_{L^\infty (\RR^d)} +\|\nabla^{\ell+1} f\|_{L^\infty(\RR^d)}.$$
     \begin{rem}
     These estimates are already reformulated slightly to suit the
     analysis in this paper below, and in particular, here and in what
     follows,  the notation $\nabla^k f$  denotes the vector with
     components $D^\a f$, $|\a|=k$, for a given positive integer
     $k$. For convenience, we also set $\nabla^0 f=f$.
\end{rem}
\begin{rem}\label{rem-2-3}
Due to the slightly faster decay \eqref{1-1}, we can easily  prove, whenever $b=0$, that neither the log term in \eqref{0-3}, nor the term $\|\nabla^{\ell} f\|_{L^\infty (\RR^d)}$ in $C_f$    is needed.
\end{rem}

      Using \eqref{1-1}, and a straightforward
      calculation, we can deduce the following useful estimates  which
      shall be employed in the next section:
      \begin{prop}\label{prop-1-1} Let $A$ be a quasi-uniformly distributed subset of $\RR^d$, and let $h\in (0,1)$.  Then for $0\leq a<m_2$,
      \begin{equation}\label{1-4-i}\sup_{x\in\RR^d} \sum_{\a\in A} |\psi_\a (x/h)| (\|x-\a h\|+h)^a \leq C_a h^a,\end{equation}
      and for  any $\va>0$,
      $$\sup_{x\in\RR^d} \sum_{\a\in A} |\psi_\a (x/h)| (\|x-\a h\|+h)^{m_2} (1+\|x-\a h\|)^{-\va}\leq C_\va h^{m_2}|\log h|.$$
      \end{prop}

\section{Pointwise approximation by  quasi-interpolation for functions of lower smoothness  in $L^p$}
\setcounter{equation}{0}

Error estimates where the approximands are rougher, i.e.\ from larger
function spaces with less smoothness required, than the approximating
(radial) basis functions or quasi-interpolants are essential in practical applications
because we cannot in general guarantee high order smoothness of the
possibly unknown functions which are to be approximated (i.e., they
are only known e.g.\ at a finite number of points).
On the other hand, however,   functions of  lower smoothness in $L^p$  (say,  for  instance, $f\in W_p^r$ with $r<d/p$)  are defined only almost everywhere on $\RR^d$,
and pointwise evaluation functionals make no sense for such
functions. Thus,  the quasi-interpolation  operators $Q_h$ considered
in the last two sections  are  no longer  applicable to functions of
lower smoothness in $L^p$. In  this section, we will modify
  the definition of    $Q_h f$  so that it is defined for every locally integrable function.  It turns out that
 such  modified operators, although not as convenient for
 practical  implementation,  is easier to handle.  In fact, we will
 establish stronger pointwise error estimates in this section, which
 will play crucial  roles in Section 4 when we consider compression in  approximation by radial basis functions.

 We start with the definition of the modified quasi-interpolation.
 First, consider the linear functional $S_0: \Pi_\ell^d\to \RR$ given by $S_0 P=P(0)$.
Since the space $\Pi_\ell^d$ is  finite dimensional, it follows that
\begin{equation*}|S_0 P|\leq \max_{y\in B_1(0)} |P(y)|\leq C_{\ell,d} \int_{B_1(0)} |P(y)|\, dy, \   \      \forall P\in\Pi_\ell^d.\end{equation*}
By the Hahn-Banach theorem,  this linear functional can be extended to
a bounded linear functional on the space $L^1(B_1(0))$  such that $S_0 P=P(0)$ for all $P\in\Pi_\ell^d$, and
\begin{equation}\label{4-1}|S_0 f|\leq C  \int_{B_1(0)}|f(y)|\, dy,\   \ \forall f\in L^1(B_1(0)).\end{equation}
For each $\a\in A$, we define
$$S_\a f:= S_0 \Bigl( f(\cdot+\a)\Bigr),\   \   f\in L^1(B_1(\a)).$$
 Hence, $S_\a P=P(\a)$ for all $P\in\Pi_\ell^d$, and by \eqref{4-1},
\begin{equation}\label{4-2-0}|S_\a f|\leq C \int_{B_1(\a)} |f(y)|\, dy,\   \  \forall f\in L^1(B_1(\a)),\   \  \forall \a\in A.\end{equation}
Now for $h>0$ and every locally  integrable function $f$ on $\RR^d$,  we define the modified quasi-interpolation  $\wt{Q}_h f$  by
\begin{equation}\label{4-2}\wt{Q}_h f(x):=\sum_{\a\in A}
f^h(\a)  \psi_\a (x/h),\  \ x\in\RR^d,\end{equation}
where $f^h(\a):=  S_{\a } (\s_h f)$ and $\sa_h f(\cdot)=f(h\cdot)$.

There are many smoothing techniques to construct the coefficients $f^h(\a)$  in \eqref{4-2} explicitly via integrals. 
Next, we give several examples of these techniques. 

 \begin{exam}Let $\{ \phi_j\}_{j=0}^{\text{dim}\, \Pi_\ell^d}$ be  an  orthonormal basis of $\Pi_\ell^d$ with respect to the inner product of $L^2(B_1(0))$, and let
 $\Phi_\ell (y)=\sum_{j=0}^{\text{dim}\, \Pi_\ell^d} \phi_j(y) \phi_j(0)$.  Define
$S_0 f :=\int_{B_1(0)} f(y)\Phi_\ell (y)\, dy$ for   $ f\in L^1(B_1(0)). $
Accordingly, we have
$$f^h(\a):= h^{-d }\int_{B_h(\a h)} f(y)\Phi_\ell (h^{-1}y-\a)\, dy,\    \     \  h>0, \    \   \  \a\in A.$$
\end{exam}

\begin{exam}
 Let $N\in C^{2\ell+1}(\RR^d)$ be a compactly supported multivariate spline  with $\text{supp} \, N \subset B_{c_1}(0)$ for some $c_1>0$ such that
$$\sum_{j\in\ZZ^d} p(j) N(x-j)=p(x),\   \  \forall p\in \Pi_\ell^d.$$
The Strang-Fix condition on $N$ then  implies that (see, for instance, \cite{BDL})
$$p(0)=\int_{\RR^d} p(x) N(x)\, dx,\   \  \forall p\in\Pi_\ell^d.$$
Thus, we can define
$$S_0(f):=\int_{B_1(0)} f(x) N_{c'}( x)\, dx,\   \  \forall f\in L^1(B_1(0)),$$
where  $c'=c_1^{-1}$ and  $N_{\va}(\cdot)=
\va^{-d} N(\va^{-1}\cdot)$ for $\va>0$.
Accordingly,
$$f^h(\a):= h^{-d }\int_{B_h(\a h)} f(y) N_{c'}(h^{-1}y-\a)\, dy,\    \     \  h>0, \    \   \  \a\in A.$$

\end{exam}

The  next example is  due to  Jia and Lei  \cite{JL}, who also obtained the  order of $L^p$-approximation  by the associated quasi-interpolation
with   $A=\ZZ^d$.
\begin{exam}
 Let  $\eta\in C_c^\infty(\RR^d)$ be such that $\text{supp}\  \eta \subset B_1(0)$  and $\int_{\RR^d}\eta(x)\, dx =1$.
 Define
 $$S_0 f:=\int_{B_1(0)} ( f(0) -\tr_u^{\ell+1} f(0) )\eta(u)\, du,$$
 where $\tr_u^k$ denotes the $k$-th order difference operator with  step $u$.  Since $ \tr_u^{\ell+1} P=0$ for all $P\in \Pi_\ell^d$, $S_0$ clearly satisfies the conditions stated above.
 In this case,  the $f^h$  can be   expressed   explicitly as   in
 (5.1) in \cite{BD1} with $k=\ell+1$: viz.
\begin{equation}
f^h(x) :=h^{-d}\int_{\RR^d} ( f(x) -\tr_u^k f(x) )\eta(u/h)\, du, \  \ x\in \RR^d,\end{equation}
where $\tr_u^k$ denotes the $k$-th order difference operator given by
$$\tr_u^k f(x): =\sum_{j=0}^k (-1)^j \binom{k}{j} f(x-ju),\   \ x\in\RR^d.$$

 \end{exam}

The main goal in this section is to prove the following pointwise error estimate:

  \begin{thm}\label{thm-6-4-18}Let
 $f\in C^r(\RR^d)$,   $1\leq r\leq \ell+1$  and $x\in\RR^d$.
  \begin{enumerate}[\rm (i)]
 \item If $r < m_2$, then for all $h>0$,
 \begin{equation}\label{6-4-18}
 |f(x)-\wt{Q}_hf(x)|\leq C
    h^r\bl( \|\nabla^r f\|\ast \Psi_h\br)(x)
    \end{equation}
    where
    $\Psi(z):=\|z\|^{-d+1}(1+\|z\|)^{-1-m_2+r}$
  and $\Psi_h(\cdot)=h^{-d} \Psi(h^{-1}\cdot)$.
  \item  If $r = m_2$, then for  $h\in (0,1)$,
  \begin{align}
   |f(x)-\wt{Q}_hf(x)| \leq &C h^r\Bl[(\|\nabla^r f\|\ast \Phi^h)(x) +   \|\nabla ^{r-1} f(x)\|\Br]\label{6-5-12}
      \end{align}
  where
  $$\Phi^h (z):=\|z\|^{-d+1} (\|z\|+h)^{-1} (1+\|z\|)^{-1}.$$


 \end{enumerate}

 \end{thm}
A straightforward   calculation shows that  the convolutions on the
right hand sides of  both \eqref{6-4-18}  and \eqref{6-5-12}  can be
dominated by the Hardy-Littlewood maximal functions of some
derivatives of $f$. Thus, from Theorem \ref{thm-6-4-18},  we deduce
the following useful  corollary, which  will play a crucial role in
Section 4 when  we consider compressions  in  radial basis function
approximation. The operator $M$ used here is the Hardy--Littlewood
maximal function 
$$ Mf(x):=\sup_{\delta>0}{1\over|B_\delta(x)|}\int_{B_\delta(x)}|f(z)|\,dz.$$

\begin{cor}\label{cor-6-5-18}
Let  $r\in\NN$, $1\leq r\leq \ell+1$,  $f\in C^r(\RR^d)$, and $x\in\RR^d$.
If $r<m_2$, then
\begin{align}\label{}
    &\   \    \  |f(x)-\wt{Q}_hf(x)|\leq C h^r      M(\|\nabla^r f\|)(x), \    \    \   h>0,\label{4-15}
\end{align}
whereas if $r=m_2$, then for $h\in (0,1)$,
\begin{align*}\label{}
    &\   \    \  |f(x)-\wt{Q}_hf(x)|\leq C h^r
   \max\{ 1, |\log h|\}\Bigl(M(\|\nabla^r f\|)(x)+ \|\nabla^{r-1} f(x)\|\Bigr). \end{align*}
\end{cor}

 Corollary \ref{cor-6-5-18}  together with the usual density argument  also  implies  the following rate of pointwise convergence of  the modified quasi-interpolation.

\begin{cor}If  $1\leq p<\infty$, $w\in A_p$, $f\in W_p^r(\RR^d,w)$ and $1\leq r\leq \ell$, then for almost every $x\in\RR^d$,
$$\lim_{h\to 0} h^{-r}(f(x)-\wt{Q}_h f(x))=0.$$
\end{cor}

Another result  that can be deduced directly  from Theorem \ref{thm-6-4-18} via  straightforward calculations
 is  the following.

\begin{cor} \label{cor-6-6}Let   $1\leq p\leq \infty$,   $w\in A_p$,  $1\leq r\leq \ell+1$,  and  $f\in W_p^r(\RR^d,w)$.
Then for $h\in (0,1)$,
\begin{align}\label{6-7-18}
\|f-\wt{Q}_h f\|_{p,w} \leq  Ch^r\|f\|_{W_p^r(\RR^d, w)},
\end{align}
with $h^r|\log h|$ in place of $h^r$ in the case when $r=m_2$.\end{cor}

%

The rest of the section is devoted to the proof of Theorem \ref{thm-6-4-18}.

\subsection{Proof of Theorem \ref{thm-6-4-18}}

For  simplicity, we will prove the result  for the case  of $1\leq r<m_2$ only.  The proof below with slight modifications works equally well for
the  case  of $r=m_2$.
By the property of   polynomial reproduction,  we  may write
\begin{align}
    |f(x)-\wt{Q}_h f(x)|\leq \sum_{\a\in A} |S_\a (\sa_h f) -T_x^{r-1} f(\a h)||\psi_\a (x/h)|.\label{6-8-18}
\end{align}
Since $T_x^{r-1} f(\a h) =T_{x/h}^{r-1} (\sa_h f)(\a)$  and $T_{x/h}^{r-1} (\sa_h f) \in \Pi_\ell^d$,  it follows   by \eqref{4-2-0}  that
\begin{align}
    |S_\a (\sa_h f) -T_x^{r-1} f(\a h)|&=\Bigl|S_\a \Bl[ \sa_h f-T_{x/h}^{r-1} (\sa_h f)\Br]\Br|\notag\\
    &\leq C \f 1{|B_h(\a h)|}\int_{B_h(\a h)} |f(y)-T_x^{r-1} f (y)|\, dy, \label{4-8}
\end{align}
which, using Theorem 3.1 of \cite{BD1} with $r=m$,  is bounded above by
\begin{align}C h^{-d} &\int_{B_h(\a h)}\|x-y\|^{r-1} \max_{u=x,y}  \Bl[\int_{B_{ 2\|x-y\|}(x)} \|\nabla^r f(z)\|\, \|z-u\|^{1-d}\, dz\Br]\, dy.\label{6-10-18}\end{align}
   Since $B_{2\|x-y\|}(x)\subset B_{2(\|x-\a h\|+h)}(x)$ whenever  $y\in B_h(\a h)$, the term  in \eqref{6-10-18} is bounded above by a constant multiple of
   $$ ( h+\|x-\a h\|)^{r-1} L_{x,\a, h}(x)+h^{-d}\int_{B_h(\a h)} \|x-y\|^{r-1} L_{x,\a,h} (y)\, dy,$$
   where
   \begin{equation}\label{6-11-20}L_{x,\a,  h}(u):= \int_{B_{2(\|x-\a h\|+h)}(x)} \|\nabla^r f(z)\|\|z-u\|^{-d+1}\, dz.\end{equation}
   Thus,
   the term on the right hand side of \eqref{6-8-18} is dominated  by a constant multiple of  $\Sigma_1(x)+\Sigma_2(x)$, where
\begin{align*}
\Sigma_1(x):=&\sum_{\a\in A   } \bl( h+ \| x -\a h\|\br)^{r-1}|\psi_\a(x/h)|  L_{x, \a, h} (x),\\
 \Sigma_2(x):=&h^{-d}  \sum_{\a\in A}|\psi_\a (x/h)|\int_{B_h(\a h)}  \|x-y\|^{r-1} L_{x, \a, h} (y)\, dy.\end{align*}

By the proof of Lemma 4.4 in \cite{BD1} with $m=r$, it is easily seen that
\begin{equation}\label{6-12-20}
\Sigma_1(x)\leq C h^r \bl( \|\nabla^r f\|\ast \Psi_h \br)(x).
\end{equation}

To estimate  the  sum $\Sigma_2(x)$, we decompose $A$ as $A:=\cup_{j=0}^{\infty} A_{j,h}(x)$, where
$A_{ 0, h}(x):=\{ \a\in A:\  \  \|\a-h^{-1} x\|\leq 2\}$  and $
    A_{j, h}(x):=\{ \a\in A:\  \ 2^j  <\|\a-h^{-1}x\|\leq 2^{j+1}\}$  for $  j\ge 1$.
Using \eqref{6-11-20},  for each $\a\in A_{j,h}$, we have
\begin{align*}
&\int_{B_h(\a h)}  \|x-y\|^{r-1} L_{x, \a, h} (y)\, dy\\
&\leq C  (2^j h)^{r-1}\int_{B_{2^{j+3}h}(x)} \|\nabla^r f(z)\|\Bl[\int_{B_h(\a h)}\|z-y\|^{-d+1}\, dy\Br]\, dz.\end{align*}
Since
$\cup_{\a\in A_{j,h}}B_h(\a h)\subset B_{ 2^{j+2}h}(x)$ and by quasi-uniformity,
$$\sup_{x\in\RR^d}\sup_{h\in (0,1)} \sum_{\a\in A} \newchi_{B_h(\a h)}(x)\leq C<\infty,$$
it follows that
\begin{align*}
   &h^{-d}\sum_{\a\in A_{j,h}(x)}|\psi_\a (x/h)|\int_{B_h(\a h)}  \|x-y\|^{r-1} L_{x, \a, h} (y)\, dy
\\&\leq C
   2^{-jm_2} (2^j h)^{r-1-d}\int_{B_{2^{j+3}h}(x)} \|\nabla^r f(z)\|\Bl[\int_{B_{2^{j+2}h}(x)}\|y-z\|^{-d+1}\, dy\Br]\, dz\\
   &\leq C 2^{-jm_2} (2^j h)^{r-d}\int_{B_{2^{j+3}h}(x)} \|\nabla^r f(z)\|
   dz,\end{align*}
where the second step uses the fact that $B_{2^{j+3}h}(x)\subset B_{2^{j+4}h}(z)$ whenever $z\in B_{ 2^{j+3}h}(x)$.
Hence, by Fubini's theorem,
\begin{align}
 &\Sigma_2(x)\leq C \sum_{j=0}^\infty2^{-jm_2} (2^j h)^{r-d}\int_{B_{2^{j+3}h}(x)} \|\nabla^r f(z)\|
   dz     \notag  \\
   &\leq C h^{r-d} \int_{\RR^d}\|\nabla^r f(z)\| \Bl[\sum _{2^j \ge c' (1+h^{-1}\|x-z\|)} 2^{-j(m_2+d-r)}\Br]\, dz \notag\\
   &\leq C h^r(\|\nabla^r f\|\ast \Psi_h )(x).  \label{6-13-20}
 \end{align}
Thus, combining \eqref{6-12-20} with \eqref{6-13-20}, we obtain the desired estimate,  and hence complete the proof.

\section{Compression in the space $C(\RR^d)$}
\setcounter{equation}{0}

One view of using compression for our approximation purposes is to
limit the number of terms in the expansions with radial basis
functions.  Compression with radial basis functions in the very
special way of using Gau\ss-kernels turns up also in the paper
\cite{HR}.

Indeed,
let $\vi$ be a continuous  function on $\RR^d$, and let $A\subset \RR^d$ be quasi-uniformly distributed.
For a positive integer $n$, we define
    $$\Sigma_{n,\vi, h} :=\Bigl\{ \sum_{j=1}^n c_j \vi ( h^{-1} \cdot -\a_j):\   \  \a_j\in A, \  x_j \in\RR^d,\
     \  c_j\in\RR\Bigr\},\   \  h\in (0,1).$$
    Our main purpose in this section
    is to study the following $n$-term approximation of a given continuous function $f$:
    $$ \hat\sa_{n,h} (f):=\inf_{g\in \Sigma_{n,\vi,h}}\|f-g\|_\infty.$$

 We start with a negative result for compression in the space $C(\RR^d)$, which asserts that  given  any $h>0$,
 the quantity $\hat\sa_{n,h} (f)$ may not tend
  to zero as $n\to \infty$ for  general
   continuous functions $f$.

    \begin{thm}
    Let $\{\psi_j\}_{j\in\ZZ}$ be a sequence of    continuous functions on $\RR^d$ with
         $$\lim_{\|x\|\to\infty} |\psi_j(x)|=0,\   \  \forall j\in\ZZ.$$ Then for any $n\in\NN$, and $r\in\NN$,
    $$ \sup_{\|f\|_{W_\infty^r}\leq 1 } \inf_{g\in \Sigma_n} \|f-g\|_{L^\infty(\RR^d)} \ge c_r>0,$$
    where
        $$\Sigma_n:= \Bigl\{ \sum_{j=1}^n c_j \psi_j(h_j^{-1} \cdot):\  c_j\in \RR, \  \  h_j>0\Bigr\}.$$

    \end{thm}

    \begin{proof}
Since for
     any $g\in \Sigma_n$,
    $$\lim_{\|x\|\to \infty} g(x)=0,$$
    it follows that
    $$\|f-g\|_\infty $$ 
remains positive for any constant function $f$ or infinite sum $f$ of shifts
of a suitable compactly supported test-function.   \end{proof}

The above proposition shows that in order to study the $n$-term approximation problem in $C(\RR^d)$, one has to assume some decaying condition of the function at infinity. For example, we have a positive result if we assume that $f\in W_p^r$ with $r>\f dp$ for some  $1\leq p<\infty$.
In fact, a greedy algorithm can be constructed in this case. This is the theme of our next theorem.
We refer to the book \cite{Te} for recent development
of greedy  approximation.

\begin{thm}\label{prop-4-2}
Assume that   $1\leq p<\infty$, $r>\f dp$ and
$f\in W_p^{r}(\RR^d)$. Let $\be: =\min\{ r-\f dp,\ell+1\}$.  For $n\in\NN$ define
$$\mathcal{G}_n f(x)=\sum_{\a\in \Ld} f(h\a)
 \psi_\a(h^{-1} x)\  \ \  \text{with $ h\sim n^{-\f1{p\b +d}}$},$$
where $\Ld\subset A$ satisfies that $\# \Ld = n$, and
$\min_{\a\in \Ld} |f(h\a)|
 \ge \max_{\a\in A\setminus\Ld} |f(h\a)|$. Then
\begin{equation}\|f-\mathcal{G}_n f\|_{L^\infty(\RR^d)}
\leq C n^{-\f{\b}{p\b+d}}\log n \|f\|_{W_p^r(\RR^d)}.\end{equation}\end{thm}

The proof of Theorem  \ref{prop-4-2} relies on the following straightforward  lemma.

  \begin{lem}\label{lem-4-3} Given any $(a_\a)_{\al\in A}\subset \RR$, we have $$\Biggl\|\sum_{\a\in A }a_\a \psi_\a\Biggr\|_{L^p(\RR^d)} \leq C \Bigl(\sum_{\a\in A}|a_\a|^p\Bigr)^{\f1p},\   \  1\leq p\leq \infty,$$
  with the usual change when $p=\infty$.\end{lem}
  \begin{proof}This lemma follows directly from the H\"older
    inequality and the first estimate in Proposition~2.3 in
    \cite{BD1} (i.e., \eqref{1-4-i}) with $i=0$.
  \end{proof}

We are now in a position to prove Theorem  \ref{prop-4-2}.\\

{\it Proof of Theorem \ref{prop-4-2}.}
Using Lemma 6.4 of \cite{BD1}, we have
$$\sum_{\a\in A} |f(h\a)|^p \leq C^p h^{-d} \|f\|^p_{W_p^r}.$$
This, in particular, implies that  there are  at most $n$ elements in the set $A$ such that
\begin{equation*} |f(h\a)|\ge C h^{-d/p} n^{-\f1p} \|f\|_{W_p^r}.\end{equation*}
Thus, for the set $\Ld$ defined in Theorem  \ref{prop-4-2}, we have
$$|f(h\a)|\leq  C h^{-d/p} n^{-\f1p} \|f\|_{W_p^r},\   \ \forall \a\in A\setminus \Ld.$$
By  Lemma \ref{lem-4-3}, this in turn implies that
$$|Q_hf(x)-\mathcal{G}_n f(x)|\leq
\sum_{\a\in A\setminus\Ld} |f(\a h)||\psi_\a(x/h)|\leq Ch^{-d/p} n^{-\f1p} \|f\|_{W_p^r}.$$
On the other hand, using (5.7) in \cite{BD1},  $$\|f-Q_h f\|_\infty\leq C h^\b |\log h|\|f\|_{W_p^r}.$$
Therefore,  choosing  $h\sim n^{-\f 1{\b p+d}}$ so that
$ h^\b\sim n^{-\f1p}h^{-d/p},$
we have that
$$\|f-\mathcal{G}_n
f\|_\infty\leq C n^{-\f \b{\b p+d}} \log n \|f\|_{W_p^r}.$$
\hb

In general, we have the following result:

\begin{prop}
    Assume that $1\leq r\leq \ell+1$, and $f: \RR^d\to \RR$ has bounded, continuous partial derivatives up to the order $r$.   Assume, in addition,  that  $f$ satisfies the decaying condition
    \begin{equation}\label{rem2}
      \sup_{M\ge 1} M^{a}\Bl(\int_{\|x\|\ge M }|f(x)|^p\, dx\Br)^{\f1p}\leq C_1<\infty
    \end{equation}
    for some $a\ge 0$, and $1\leq p< \infty$. For $n\in\NN$, let $h\in (0,1)$ be such that
    $$h\sim  n^{-\f 1d \f{ap+d}{d+(a+r)p}}, $$  and define
    $$\Ld:=\Bl\{\a\in A:\  \|\a\|\leq n^{1/d}\
    \text{\rm or}\  \  |S_\a (\sa_h f)| \ge h^r \Br\}.$$
    Then $\#\Ld\leq cn$, and the function
    $$\mathcal{G}_n f(x):=\sum_{\a\in \Ld} (S_\a (\sa_h f)) \psi_\a(x/h)$$
    satisfies
    \begin{equation}\label{rem}\|f-\mathcal{G}_n f\|_\infty\leq C n^{-\f rd \f {a+d}{a+d+r}}\Bl[ C_1 + \|\nabla^r f\|_\infty +\|\nabla^\ell f\|_\infty\Br].\end{equation}
\end{prop}

We note that for compactly supported functions $f$, \eqref{rem2} is true for any positive number $a$, with the constant $C_1$ depending on the diameter of the support of $f$, and \eqref{rem} then provides the optimal approximation order.

\begin{proof} Using \eqref{4-1}, we have that
\begin{align*}
    h^d\sum_{\|\a\|\ge n^{1/d}} |S_\a (\sa_h f)|^p  &\leq C \sum_{\|\a\|\ge n^{1/d}}  \int_{B_h(\a h)} |f(y)|^p\,  dy\\
    &\leq C \int_{\|y\|\ge  n^{1/d}h/2} |f(y)|^p  dy\leq C C_1^p (n^{1/d} h)^{-ap}.
\end{align*}
This means that  there exist at most $c n$ points $\a$ in the set $A$ such that $\|\a\|\ge n^{1/d}$ and
 $$|S_\a (\sa_h f)|  \ge C^{1/p} C_1 n^{-\f ad -\f1p} h^{-a-\f dp}.$$
Since
$$\# \Bl\{\a\in A:\ \|\a\|\leq n^{1/d}\Br\}\leq cn,$$
it follows that $\# \Ld\leq c'n$, and that
\begin{equation}\label{}
 |S_\a (\sa_h f)|  \leq  C^{1/p} C_1 n^{-\f ad -\f1p} h^{-a-\f dp}\leq C C_1 h^r ,\   \  \forall \a\in A\setminus\Ld.
\end{equation}
Thus,
\begin{align*}
    \|f-\mathcal{G}_n f\|_\infty&\leq \|f-Q_h f\|_\infty+\|Q_h f-\mathcal{G}_n f\|_\infty\\
    &\leq C_f  h^{r}|\log h| + C \sup_{\a \in A\setminus \Ld} |S_\a (\sa_h f)|\\
    &\leq C_f h^r |\log h|+ C C_1 h^r\leq C_f' n^{-\f rd \f {a+d}{a+d+r}}.
\end{align*}

\end{proof}

\section{Compression in $L^p$ with $1<p<\infty$}
\setcounter{equation}{0}
 The main purpose in this section is to study compression with  radial basis functions  for smooth functions in $L^p$.
 We will follow closely the  argument   of  DeVore and Ron \cite{DR} .  The pointwise error estimates (i.e.,  Corollary 4.2 in \cite{BD1}  and Corollary \ref{cor-6-5-18}) will play crucial roles in our proof.

We begin with studying decompositions of locally integrable functions by wavelets (see the aforementioned book by Daubechies). Later-on we shall use these remarks for studying radial basis functions and quasi-interpolants.

\subsection{Wavelet decomposition  and  the Triebel-Lizorkin spaces}   Most materials in this subsection can be found in \cite{M} and \cite{DR}.
Let $\DD_j$ denote the set of dyadic cubes of side length $2^{-j}$, and let  $\DD=\bigcup_{j\in\ZZ}\DD_j$. Thus, each $I\in\DD_j$ can be written in the form
$$ I=2^{-j} ( k+[0,1)^d)  \    \   \text{for some $  k\in\ZZ^d$}.$$
Write $\ell(I)=2^{-j}$ for each $I\in\DD_j$.
Let $$E:=\{ 1, 2, \ldots, 2^d-1\}, \   \   \mathcal{V}=\DD\times E.$$
For each $v=(I_v, e_v)\in \VV$ with $I_v\in\DD$ and $e_v\in E$, we denote
$$\ell(v):=\ell(I_v),\   \ |v|=\ell(v)^d=|I_v|.$$

Let $\{w_v\}_{v\in\VV}$ be an orthogonal  wavelet basis for $L_2(\RR^d)$  with the following  properties:
\begin{enumerate}[\rm (i)]
\item Each wavelet $w_v$ with $I_v:=2^j (k+[0,1]^d)\in\DD$, is supported in the cube
    $$\bar{I}_v =2^j (k+A_0 [0,1]^d),$$
    with $A_0$ being some fixed absolute constant.

    \item The  wavelets $w_v$ are $m_2$-times differentiable, and have $m_2$ vanishing moments  for a sufficiently large positive integer  $m_2$. Moreover,
        \begin{equation}\label{6-1}|D^\a w_v(x)|\leq C \ell(I_v)^{-|\a|}\newchi_{\bar{I}_v}(x),\   \  |\a|\leq m_2,\   \  x\in\RR^d,\end{equation}
        and
    $$\|w_v\|_p\sim |I_v|^{\f1p},\   \   \  \forall v\in\VV,\   \ 1\leq p\leq \infty.$$

    \item Each locally integrable function $f$ defined on $\RR^d$ has a wavelet decomposition
        $$ f=\sum_{v\in\VV} f_v w_v,\  \  f_v:=|I_v|^{-1}\la f, w_v\ra,$$
        and many classical function spaces on $\RR^d$ can be defined or characterized equivalently in terms of the coefficients of these wavelet decompositions. For example, given   $s>0$ and $0<p,q<\infty$,  the Triebel-Lizorkin spaces $F_{p,q}^s(\RR^d)$ can be equivalently defined to be  the space of all functions $f$ on $\RR^d$ with
         $$|f|_{F_{p,q}^s} := \|\Delta_{s,q} f\|_{L^p(\RR^d)}<\infty,$$
            where
            $$\Delta_{s,q} f(x):=\Bl ( \sum_{v\in\VV} \ell(v)^{-qs} |f_{v}|^q \newchi_{\bar{I}_v}(x)\Br)^{1/q},$$
             The quasi-norm in $F_{p,q}^s(\RR^d)$ is defined by
            $$\|f\|_{F_{p,q}^s(\RR^d)} :=\|f\|_{L^p(\RR^d)} +|f|_{F_{p,q}^s(\RR^d)}.$$          The definition of the Triebel-Lizorkin spaces $F_{p,q}^s$ extends naturally to the $q=\infty$ case, with $\Delta_{s,\infty} f$ defined by
            $$\Delta_{s,\infty} f(x) :=\sup_{v\in\VV} \ell(v)^{-s} |f_{v}|\newchi_{\bar{I}_v}(x).$$                  \end{enumerate}

\subsection{Nonlinear approximation}
Let $\Sigma_N=\Sigma_N(\vi)$ denote the set of all functions $S$  on $\RR^d$ for which there exists a subset $\Ld\subset A$ of cardinality $ \#\Ld=N$ such that $$ S=\sum_{\al\in\Ld} a_\a\psi_\a (h^{-1}\cdot)\  \ \   \text{for some  $h>0$ and $a_\a\in\RR$},$$
with the quasi-interpolants as above.
   Define
$$\sigma_N (f)_p:=\inf_{S\in \Sigma_N} \|f-S\|_{L^p(\RR^d)}.$$

For the sake of simplicity, we will assume that
$\ell+1<m_2$ in this section so that there is  no
 log term in \eqref{4-15}.

\begin{thm} \label{thm-6-1}Let $1< p<\infty$ be given, and let $f\in F_{\tau, q}^s$, with $0<s\leq \ell+1$, $\tau:=(1/p+s/d)^{-1}$ and $q=(1+s/d)^{-1}$. Then
$$\hat\sa_N (f)_p\leq C N^{-s/d} \|f\|_{F_{\tau,q}^s}.$$
\end{thm}

The proof of Theorem \ref{thm-6-1}  follows along the same line as
that of Theorem 6.2 of \cite{DR}. It relies on a series of  lemmas
with which we will begin.

  \begin{lem}\label{lem-6-2}For each $v\in\VV$ and $N_v\in\ZZ_+$, there exists a positive number $h\sim \ell(I_v) (N_v+1)^{-1/d}$ such that
  $$ S_{v,N_v}:= \wt{Q}_h w_v\in \Sigma_{N_v},$$
  and
 \begin{align}
 R_v(x):&=| w_v(x) -S_{v,N_v}(x)|\label{6-2-1}\\
  &\leq C ( N_v+1)^{-(\ell+1)/d} \Bl[M (\newchi_{\bar{I_v}})(x)\Br]^{1+\f {m_2}d},\  \  x\in\RR^d.\label{6-3-1}\end{align}
\end{lem}
\begin{proof} By \eqref{6-1}, the estimate holds trivially if $N_v=0$. For the rest of the proof, we assume that $N_v\ge 1$.
Let $h$ be a positive number  to be specified in a moment, and let
$$\Ld:=\Bl\{\a\in A:\  \  \|\a h-c_{v}\|\leq  C  A_0 \ell(I_v)\Br\},$$
where  $c_v$ denotes the centre of the dyadic cube $I_v$, and $C$ is a large positive constant depending only on $d$.  Then $$\# \Ld \leq C \Bl( 1+\f {\ell(I_v)} {h}\Br)^{d}.$$
Thus, we may choose $ h\sim \ell(I_v) N_v^{-1/d}$ so that
$\# \Ld\leq N_v$ and
$$\bigcup_{x\in \bar{I_v}} B_{4h}(x)\subset B_{ C  A_0\ell(I_v)}(c_v).$$
Since $w_v$ is supported in $\bar{I}_v$, it follows that $w_v$ is identically zero on $\bigcup_{\a \in A\setminus\Ld} B_{2h}(\a h)$.
 This in turn implies that
$$\ell_{\a} [(w_v)_h]=S_0[ w_v (h\cdot+h\a)]=0,\   \    \    \  \forall \a\in A\setminus\Ld,$$
and hence
$$Q_h w_v=\sum_{\a\in\Ld}\ell_{\a} [(w_v)_h] \psi_\a(h^{-1} \cdot )\in \Sigma_{N_v}.$$
Therefore, on one hand, using \eqref{4-15} and \eqref{6-1},
\begin{align}
    &|w_v(x)-\wt{Q}_h w_v(x)|\leq C h^{\ell+1} M(\|\nabla^{\ell+1} w_v\|)(x)\notag\\
    &\leq C \Biggl(\f{h}{\ell(I_v)}\Biggr)^{\ell+1} \leq C N_v^{-(\ell+1)/d} ,\  \  x\in\RR^d.\label{6-2}
\end{align}
On the other hand, if $\|x-c_v\|\ge 2CA_0\ell(I_v)$,  then $w_v(x)=0$, and
\begin{align}
 |w_v(x)-\wt{Q}_h w_v(x)|&\leq C\sum_{\a\in\Ld}|\ell_{\a} [(w_v)_h]| \Bl( 1+h^{-1}\| x -\a h\|\Br)^{-m_2-d}\notag\\
 &\leq C\sum_{\a\in\Ld} \Bl( h^{-1}\| x -c_v\|\Br)^{-m_2-d}\f 1{|B_h(\a h)|}\int_{B_h(\a h)} |w_v(z)|\, dz\notag\\
 &\leq C \Biggl( \f {\|x-c_v\|}{\ell(I_v)}\Biggr)^{-m_2-d}  \Bl( \f  { \ell(I_v)}h\Br)^{-m_2-d}\#\Ld\notag\\
 &\leq C \Biggl( 1+ \f {\|x-c_v\|}{\ell(I_v)}\Biggr)^{-m_2-d} N_v^{-m_2/d}.\label{6-3}
\end{align}
Therefore, combining \eqref{6-2} with \eqref{6-3}, and in view of the fact that $m_2\ge \ell+1$,  we deduce that
\begin{equation}\label{}
    |w_v(x)-\wt{Q}_h w_v(x)|\leq CN_v^{-\f{\ell+1}d}\Bl( 1+ \f {\|x-c_v\|}{\ell(I_v)}\Br)^{-m_2-d},\  \ x\in\RR^d.
\end{equation}
Finally, to complete the proof, we just need to observe that
$$ M(\newchi_{I_v})(x)\ge C_0 \Biggl( 1+ \f{\|x-c_v\|}{\ell(I_v)}\Biggr)^{-d},\   \    \forall x\in\RR^d. $$

\end{proof}

\begin{lem}\label{6-3-lem} Let $R_v(x)$ be as defined in \eqref{6-2-1}. Then for $1< p< \infty$,
$$\Bl\| \sum_{v\in\VV} |f_v|R_v \Br\|_{L^p(\RR^d)} \leq C \Bl\| \sum_{v\in\VV} (1+ N_v)^{-(\ell+1)/d} |f_v|\newchi_{\bar{I_v}}\Br\|_{L^p(\RR^d)}.$$
\end{lem}
\begin{proof} Using \eqref{6-3-1}, and setting $q=1+\f {m_2}d$,  we have
\begin{align*}
   \Bl\| \sum_{v\in\VV} |f_v|R_v \Br\|_{L^p(\RR^d)} &\leq C\Bl\| \Bl(\sum_{v\in\VV} |f_v| ( N_v+1)^{-\f{\ell+1}d} \Bl[M \newchi_{\bar{I_v}}(x)\Br]^{q}\Br)^{1/q}\Br\|^q_{L^{qp}(\RR^d)},
   \end{align*}
   which, using the Fefferman-Stein inequality (see \cite[p. 55, 1.3.1]{S}, is estimated by
   \begin{align*}
    C\Bl\| \sum_{v\in\VV} |f_v| ( N_v+1)^{-\f{\ell+1}d} \newchi_{\bar{I_v}}\Br\|_{L^{p}(\RR^d)}.
   \end{align*}
\end{proof}

\begin{lem}\label{lem-6-4}\cite[Lemma 6.3]{DR}  Let $\{z_j\}_{j=-\infty}^\infty$ be a sequence of nonnegative numbers with $Z:=\sum_{j\in\ZZ} z_j <\infty$. Then for any $\va>0$,
$$\sum_{j\in\ZZ} z_j \Biggl (\sum_{k=-\infty}^j z_k\Biggr)^{\va-1}\leq C_\va Z^\va,$$
where it is understood that $0\cdot \infty=0$.
\end{lem}

We are now in a position to prove Theorem \ref{thm-6-1}.\\

 {\it Proof of Theorem \ref{thm-6-1}.}\   \  The proof of Theorem \ref{thm-6-1}  follows along the same line as that of Theorem 6.2 of \cite{DR}.
  Fix a nonzero function  $f\in F_{\tau,q}^s$, and a positive integer $N$. Given $x\in\RR^d$, let $\VV_x$ denote the set of all $v\in\VV$ such that $\newchi_{I_v}(x)\neq 0$.    Define for each $x\in\RR^d$ and each $v'\in\VV$,
$$M_{q,v'}(x):=\Biggl( \sum_{v\in \VV_x: v\ge   v'} |I_v|^{-qs/d} |f_v|^q \Biggr)^{1/q},$$
where we say $v\ge v'$ if either $|I_v|>|I_{v'}|$ or $|I_v|=|I_{v'}|$ and $e_v\ge e_{v'}$.
Notice that if $x\in I_{v'}$, then
\begin{align}
M_{q,v'}(x)^q=\Biggl( \sum_{v\in \VV: I_v\supsetneqq I_{v'}} +
\sum_{v\in \VV: I_v= I_{v'}, e_v\ge e_{v'} }\Biggr)|I_v|^{-qs/d} |f_v|^q.\label{6-7-0}
\end{align}
Thus, $M_{q,v'}(x)$ is actually a constant on $I_{v'}$, which we denote by $M_{q,v'}$. Thus,
$$M_{q,v'}:=M_{q,v'}(x)\leq \Delta_{s,q}(f)(x),\   \  x\in I_{v'}.$$
Let
\begin{equation}\label{6-7}a_v=a|I_v|^q |f_v|^q M_{q,v}^{\tau-q},\   \ v\in\VV,\end{equation}
where $a$ will be specified in a moment. Since  $\tau-q> 0$ and $q=1-qs/d$, it follows that
\begin{align*}
a_v=\Bl\|a|I_v|^{-qs/d} |f_v|^q M_{q,v}^{\tau-q} \newchi_{I_v}\Br\|_1\leq \Bl\|a |I_v|^{-qs/d} |f_v|^q (\Delta_{s,q} f)^{\tau-q} \newchi_{I_v}\Br\|_1.
\end{align*}
Thus,
\begin{align*}
    \sum_{v\in \VV} a_v &\leq a \Bigl\| \Delta_{s,q} (f)^{\tau-q} \sum_{v\in\VV} |I_v|^{-qs/d} |f_v|^q \newchi_{I_v}\Bigr\|_1 \\
    &\leq a \|\Delta_{s,q} (f)^\tau \|_1= a \|\Delta_{s,q} (f) \|_\tau^\tau =a\|f\|_{F_{\tau,q}^s }^\tau.
\end{align*}
Now  choose $a$ so that $a \|f\|_{F_{\tau,q}^s }^\tau=N$,
and let $N_v$ be equal to $\lceil a_v \rceil$, the greatest  integer $\leq a_v$. Then $\sum_{v\in\VV} N_v \leq N$.
Now for each $v\in\VV$, let $S_{v,N_v}\in\Sigma_{N_v}$ be as in Lemma \ref{lem-6-2}, and let
$$S:=\sum_{v\in\VV} S_{v, N_v}f_v.$$
 Then $S\in\Sigma_N$, and using Lemma \ref{6-3-lem}, we have that
\begin{align*}
    \|f-S\|_p&=\Bl\|\sum_{v\in\VV} f_v w_v -\sum_{v\in\VV} f_v s_{v, N_v}\Br\|_p\leq \Bl\|\sum_{v\in\VV} |f_v| |w_v-S_{v, N_v}|\Br\|_p \\
     &\leq C\Bl\|\sum_{v\in\VV} (1+a_v)^{-(\ell+1)/d} |f_v|\newchi_{\bar{I_v}}\Br\|_p \leq  C
    \Bl\|\sum_{v\in\VV} (1+a_v)^{-s/d} |f_v|\newchi_{\bar{I_v}}\Br\|_p\\
    &=:C\|\sum_{v\in\VV} E_v\|_p,
\end{align*}
where
$$E_v(x):=(1+a_v)^{-s/d} |f_v|\newchi_{\bar{I_v}}(x).$$

If $a_v>0$, then using \eqref{6-7},  for $x\in \bar{I_v}$,
\begin{align}
    E_v(x)&\leq a_v^{-s/d} |f_v|\newchi_{\bar{I_v}}(x)=
    a^{-s/d} |I_v|^{-qs/d} |f_v|^{1-qs/d} M_{q,v}^{\f \tau p-q}\notag\\
    &=a^{-s/d} |I_v|^{-qs/d} |f_v|^q \Biggl( \sum_{v'\in\VV:  v'\ge v}|I_{v'}|^{-qs/d} |f_{v'}|^q \Biggr)^{\f{\tau}{pq}-1}.\label{6-9}
\end{align}
On the other hand, if $a_v=0$, then by \eqref{6-7-0} and \eqref{6-7}, we have that $f_v=0$,  $E_v(x)=0$, and \eqref{6-9} remains true under the agreement that $0\cdot \infty=0$. Thus,
using Lemma \ref{lem-6-4}, it follows that
\begin{align*}
\sum_{v\in\VV} E_v(x)& \leq C a^{-s/d}\sum_{v\in\VV}  |I_v|^{-qs/d} |f_v|^q \newchi_{\bar{I_v}}(x)\Bl( \sum_{v'\in\VV:  v'\ge v}|I_{v'}|^{-qs/d} |f_{v'}|^q \newchi_{\bar{I_{v'}}}(x) \Br)^{\f{\tau}{pq}-1}\\
&\leq C a^{-s/d}(\Delta_{s,q} (f)(x))^{\tau/p}.\end{align*}
Recalling that $a=N \|f\|_{F_{\tau,q}^s}^{-\tau}$ and $\f \tau p+\f{\tau s}d=1$, we conclude that
$$\Bl\|\sum_{v\in\VV} E_v \Br\|_p \leq C a^{-s/d} \|\Delta_{s,q} (f)\|_\tau^{\tau/p}
=C N^{-s/d} \|f\|_{F_{\tau,q}^s}.$$
This completes the proof.
\hb

\def\bi{\bibitem}

\end{document}